\documentclass[12pt]{article}
\usepackage[T2A]{fontenc}
\usepackage{amsmath,amstext}
\usepackage{amsfonts}
\usepackage{amsthm}
\usepackage{amssymb}
\usepackage{srcltx}
\usepackage[unicode]{hyperref}
\usepackage[usenames,dvipsnames]{pstricks}
\usepackage{epsfig}
\usepackage{pst-grad} 
\usepackage{pst-plot} 

\newcommand{\eps}{\varepsilon}
\renewcommand{\phi}{\varphi}
\renewcommand{\le}{\leqslant}
\renewcommand{\ge}{\geqslant}
\renewcommand{\a}{\alpha}
\renewcommand{\b}{\beta}
\newcommand{\mC}{\mathcal{C}}
\newcommand{\N}{\mathbb{N}}
\newcommand{\s}{\sigma}

\newcommand{\Leb}{\mathop{\mathrm{Leb}}}
\newcommand{\Lip}{\mathop{\mathrm{Lip}}}
\newcommand{\interior}{\mathop{\mathrm{Int}}}
\newcommand{\dist}{\mathop{\mathrm{dist}}}
\newcommand{\bbN}{\mathbb{N}}
\newcommand{\const}{\mathop{\mathrm{const}}}

\newtheorem{proposition}{Proposition}

\newtheorem{lemma}{Lemma}

\theoremstyle{definition}

\newtheorem*{remark*}{Remark}
\newtheorem*{theorem*}{Theorem}
\newtheorem{definition}{Definition}
\newtheorem{definition*}{Definition}

\title{Example of a diffeomorphism\\
for which the special ergodic theorem doesn't hold}
\author{Dmitry Ryzhov\footnote{Chebyshev Laboratory (Department of Mathematics and Mechanics, St.-Petersburg State University, Russia)}}

\begin{document}
\maketitle
\begin{abstract}
In this work we present an example of $C^\infty$-diffeomorphism of a compact $4$-manifold
such that it admits a global SRB measure but for which the special ergodic theorem doesn't hold. Namely, for this transformation there exist a continuous function $\varphi$ and a positive constant $\a$ such that the following holds: the set 
of the initial points for which the Birkhoff time averages of the function~$\varphi$ differ from its $\mu$--space average by at least~$\a$ has zero Lebesgue measure but full Hausdorff dimension.
\end{abstract}

\section*{Introduction}
Let us recall the basic concepts and the notion of the special ergodic theorem, and then state our result. For more detailed survey of the subject we refer the reader to the work~\cite{KR} and to references therein.

\subsection{SRB measures and special ergodic theorems}
Let $M$ be a compact Riemannian manifold (equipped with the Lebesgue measure), and $f:M\to M$ be its transformation with an invariant measure $\mu$. For a function $\phi\colon M\to \mathbb R$ and a point $x\in M$ we denote its $n$-th time average at $x$ by
\begin{equation}
\phi_n(x):=\frac{1}{n}\sum_{k=0}^{n-1} \phi\circ f^k(x),
\end{equation}
and the space average of $\phi$ by
\begin{equation}
\bar\phi:=\int_M \phi\,d\mu.
\end{equation}

\begin{definition}
An invariant probability measure $\mu$ is called a \emph{(global) SRB measure} for $f:M\to M$ if for any continuous function $\phi$ and for Lebesgue-almost every $x\in M$, the time averages of $\phi$ at $x$ tend to the space average of $\phi$:
\begin{equation}\label{eq:time-averages}
\lim\limits_{n\to\infty}\phi_n (x)=\bar{\phi}.
\end{equation}
\end{definition}

Taking a continuous test-function $\phi\in C(M)$ and any $\a\ge 0$, we define the set of ($\phi,\a$)-nontypical points as
\begin{equation*}
	K_{\phi,\a}:=\left\{ x\in X\colon \varlimsup\limits_{n\to\infty} |\phi_n(x)-\bar\phi|>\a \right\}.
\end{equation*}

By definition, if $\mu$ is a global SRB measure, then $\Leb(K_{\phi,0})=0$.

\begin{definition}\label{def:spec-erg}
Let $\mu$ be a global SRB measure of $f$. Say that the \emph{special ergodic theorem} holds for $(f,\mu)$, if for every continuous function $\varphi\in C(M)$ and every $\a>0$ the Hausdorff dimension of the set $K_{\phi,\a}$ is strictly less than the dimension of the phase space:
\begin{equation}\label{eq:set}
	\forall \varphi\in C(M), \alpha>0 \qquad \dim_H K_{\phi,\a}<\dim M.
\end{equation}
\end{definition}

Our interest to the special ergodic theorem is related to its possible applications for studying perturbations of skew products, see, for example,~\cite{IKS}.

In~\cite{IKS} the special ergodic theorem was proved for a doubling map of a circle, in~\cite{Saltykov} for linear Anosov diffeomorphisms of two-dimensional torus, and in~\cite{KR} for all transformations for which the so-called dynamical large deviation principle holds (in particularly, for all $C^2$-uniformly hyperbolic maps with a transitive attractor).

\subsection{The counterexample}
All known sufficient properties that are required from dynamical systems to satisfy the special ergodic theorem (SET) are quite restrictive. Thus one may expect that the SET holds not for every system. The aim of the present work is to present such an example.

\begin{theorem*}
There exists a $C^\infty$-diffeomorphism of a compact $4$-manifold 
such that it admits a global SRB measure, but the special ergodic theorem doesnt'n hold for it.
\end{theorem*}


Our idea is to start with a set that has zero Lebesgue measure and full Hausdorff dimension, and then try to construct a transformation that has it as a set of $(\phi,\a)$-nontypical points for some test-function $\phi$ and some $\a>0$. For this purpose, we first describe a family of subsets of an interval that have zero Lebesgue measure and Hausdorff dimension $1$. We can construct only a discontinuous transformation such that the set from this family is $(\phi,\a)$-nontypical, so we increase the dimension from $1$ to $4$, consequently handling out the lack of continuity and smoothness of the map. This is done in the next four sections.

One can easily verify that the construction fails for typical perturbation of the system, and our example has infinite codimension in the space of all $C^\infty$-diffeomorphisms of the manifold. The existence of an open set of diffeomorphisms that do not satisfy the SET is a challenging open problem.

\subsection{Acknowledgements}
The author is very grateful to V.~Kleptsyn for sharing useful ideas and stimulating discussions. Despite of his impact, V.~Kleptsyn refused to be one of the authors of the paper. The author is also grateful to I.~Binder who taught him the method of regular measures, to Yu.~Ilyashenko for his advertence to the work and for numerous valuable comments, and to A.~Kustarev for his kind consultations about connected sums. The author is thankful to Cornell University for hospitality.

The author was supported by Chebyshev Laboratory (Department of Mathematics and Mechanics, St.-Petersburg State University) under the Russian Federation government grant 11.G34.31.0026, and partially by RFBR grant 10-01-00739-a and joint RFBR/CNRS grant 10-01-93115-CNRS-a.


\section{Dimension 1: discontinuous map of an interval}
For every sequence $P=\{p_n\}\in (0,1)^{\mathbb N}$ consider the Cantor set $C_P$, obtained from the interval $I=[0,1]$ by the standard infinite procedure of consecutive deleting the middle intervals. Namely, on the step number $n$ we take all the intervals that were obtained as a result of the previous steps (\textit{intervals of rank $n$}), and delete their central parts of relative length $p_n$, see Figure~1. In a particular case, when $p_n=1/3 \; \forall n\in\mathbb N$, we obtain the standard ''middle third'' Cantor set.

\begin{center}
\scalebox{1.2} 
{
\begin{pspicture}(0,-1.108125)(8.582812,1.108125)
\psline[linewidth=0.04cm](0.3009375,-0.0903125)(8.300938,-0.0903125)
\psdots[dotsize=0.16](0.3009375,-0.0903125)
\psdots[dotsize=0.16](8.280937,-0.0903125)
\psdots[dotsize=0.16](3.3009374,-0.0903125)
\psdots[dotsize=0.16](5.3009377,-0.0903125)
\psdots[dotsize=0.16](1.5009375,-0.0903125)
\psdots[dotsize=0.16](2.1009376,-0.0903125)
\psdots[dotsize=0.16](6.5009375,-0.0903125)
\psdots[dotsize=0.16](7.1009374,-0.0903125)
\usefont{T1}{ptm}{m}{n}
\rput(4.3723435,0.1396875){$p_1$}
\psdots[dotsize=0.16](0.7609375,-0.0903125)
\psdots[dotsize=0.16](1.0409375,-0.0903125)
\psdots[dotsize=0.16](2.5609374,-0.0903125)
\psdots[dotsize=0.16](2.8409376,-0.0903125)
\psdots[dotsize=0.16](5.7609377,-0.0903125)
\psdots[dotsize=0.16](6.0409374,-0.0903125)
\psdots[dotsize=0.16](7.5609374,-0.0903125)
\psdots[dotsize=0.16](7.8409376,-0.0903125)
\psline[linewidth=0.08cm](1.1009375,-0.0903125)(1.5009375,-0.0903125)
\psline[linewidth=0.08cm](2.1409376,-0.0903125)(2.5409374,-0.0903125)
\psline[linewidth=0.08cm](2.8209374,-0.0903125)(3.2609375,-0.0903125)
\psline[linewidth=0.08cm](5.3009377,-0.0903125)(5.7009373,-0.0903125)
\psline[linewidth=0.08cm](6.0609374,-0.0903125)(6.4609375,-0.0903125)
\psline[linewidth=0.08cm](7.1009374,-0.0903125)(7.5409374,-0.0903125)
\psline[linewidth=0.08cm](7.8809376,-0.0903125)(8.220938,-0.0903125)
\psline[linewidth=0.08cm](0.3009375,-0.0903125)(0.7009375,-0.0903125)
\usefont{T1}{ptm}{m}{n}
\rput(4.3723435,-0.8803125){$p_2$}
\usefont{T1}{ptm}{m}{n}
\rput(4.3723435,0.9196875){$p_3$}
\psline[linewidth=0.04cm,arrowsize=0.05291667cm 2.0,arrowlength=1.4,arrowinset=0.4]{<-}(0.8809375,0.0296875)(4.1009374,0.8896875)
\psline[linewidth=0.04cm,arrowsize=0.05291667cm 2.0,arrowlength=1.4,arrowinset=0.4]{<-}(7.7009373,0.0296875)(4.6809373,0.8896875)
\psline[linewidth=0.04cm,arrowsize=0.05291667cm 2.0,arrowlength=1.4,arrowinset=0.4]{<-}(2.7209375,-0.0103125)(4.1209373,0.7496875)
\psline[linewidth=0.04cm,arrowsize=0.05291667cm 2.0,arrowlength=1.4,arrowinset=0.4]{<-}(5.9009376,-0.0103125)(4.6209373,0.7096875)
\psline[linewidth=0.04cm,arrowsize=0.05291667cm 2.0,arrowlength=1.4,arrowinset=0.4]{<-}(6.7809377,-0.1503125)(4.7409377,-0.9303125)
\psline[linewidth=0.04cm,arrowsize=0.05291667cm 2.0,arrowlength=1.4,arrowinset=0.4]{<-}(1.7809376,-0.1503125)(4.0609374,-0.9303125)
\usefont{T1}{ptm}{m}{n}
\rput(0.23234375,-0.3603125){$0$}
\usefont{T1}{ptm}{m}{n}
\rput(8.272344,-0.3603125){$1$}
\end{pspicture} 
}
\\
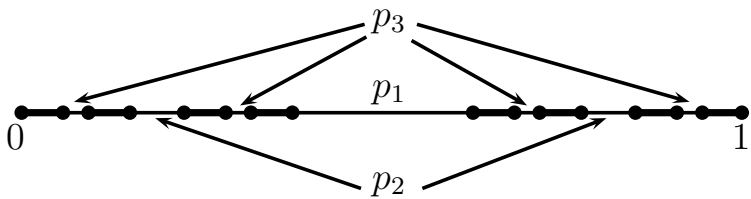
Fig. 1: first three steps of constructing the set $C_P$
\end{center}

Let $\mathcal P\subset [0,1]^{\mathbb N}$ denote the set of the sequences $P=(p_n)$
such that:

(i) $\lim_{n\to\infty} p_n=0$;

(ii) $\sum^\infty_{n=1} p_n=\infty$.


\begin{lemma}\label{1}
For every $P\in\mathcal P$ the set $C_P$ has zero Lebesgue $1$-measure and Hausdorff dimension $1$.
\end{lemma}

Though the Lemma is intuitively clear, 
we give an accurate proof. It makes us the concept of $\gamma$-regular measures, that will be used several times in our proofs.

\begin{definition}
A measure $\mu$ on a Riemannian manifold is called $\gamma$-regular, if there exist constants $c,\delta>0$ such that
\begin{equation}\label{eq:reg}
	\mu(U)<c\cdot|U|^\gamma
\end{equation}
for every measurable set $U$ with $|U|<\delta$ (here $|\cdot|$ stands for the diameter of a set).
\end{definition}

\begin{proposition}[\cite{Falc}, Mass Distribution principle]\label{mdp}
If $\text{supp }\mu\subset F$ for some probability $\gamma$-regular measure $\mu$, then $\dim_H F\ge \gamma$.
\end{proposition}

\begin{proof}
For $\eps$ small enough, $\gamma$-volume of any cover of $F$ by balls of diameter less than $\eps$ is uniformly bounded away from zero:
$$
\sum |U_i|^\gamma \ge \frac{\sum \mu(U_i)}{c}\ge \frac{\mu(F)}{c}=\frac{1}{c}>0.
$$
\end{proof}

\begin{proof}[Proof of Lemma~\ref{1}]
Let us prove that $C_P$ has Lebesgue measure zero.
Indeed, the second property $\sum^\infty_{n=1} p_n=\infty$ implies that
$$
\sum^\infty_{n=1} (-\ln (1-p_n))=\infty.
$$
Therefore $\prod^\infty_{n=1} (1-p_n)=0$ and hence $\Leb(C_P)=0$.

To prove that $\dim_H K_P=1$, consider the standard probability measure $\mu_{P}$ on $C_P$ (that is a pullback of the Bernoulli $(1/2,1/2)$-measure due to natural encoding of $C_P$ by right-infinite sequences of zeroes and ones).



Let us verify that for every $0<\gamma<1$ the measure $\mu_P$ is $\gamma$-regular. That is, we need to prove the existence of a constant $c$ such that~\eqref{eq:reg} holds for every sufficiently small interval $U$. Note that we can suppose that
$U$ is an interval of rank $n$ for some $n$. Indeed, let
\begin{equation}\label{eq:lambdan}
\lambda_n=\prod\limits_{j=1}^n \frac{1-p_j}{2}
\end{equation}
be the length of any interval of rank $n$. Then for all sufficiently large $n$ the ratio $\lambda_n/\lambda_{n+1}$ is less than $3$. Hence one can contract and shift any interval $U$ of length between $\lambda_n$ and $\lambda_{n+1}$ to an interval of rank $n+1$ changing the suitable value of a constant $c$ in~\eqref{eq:reg} by no more than $3^\gamma$.

Any interval of rank $n$ has the length $\lambda_n$ 
and the $\mu_P$-measure $\frac{1}{2^n}$. As a consequence of the first property $\lim_{j\to\infty} p_j=0$, there exists $n\in\mathbb N$ such that $2^{\gamma-1}<(1-p_j)^\gamma$ for every $j>n$. Starting with this number $n$, the sequence
\begin{equation}\label{eq:regconst}
\frac{\mu_P (U_n)}{|U_n|^\gamma}=\frac{2^{-n}}{\lambda_n^\gamma}
=\frac{2^{n(\gamma-1)}}{\left(\prod_{j=1}^n 1-p_j)\right)^\gamma}
\end{equation}
decreases, and hence is bounded. This proves the $\gamma$-regularity of $\mu_P$. Therefore, by Proposition~\ref{mdp}, $\dim_H C_P \ge\gamma$ for every $\gamma<1$. The second 
conclusion of the lemma is also proved.
\end{proof}

Of course, taking any sequence $P\in\mathcal P$ one can easily construct a discontinuous map of a unit interval for which the set $C_P$ would be $(\phi,\a)$-nontypical for some $\phi\in C(I)$ and $\a>0$. Indeed, let $f$
send the set $C_P\setminus \{1\}$ to the left endpoint $0$ and $(I\setminus C_P)\cup \{1\}$ to the right endpoint $s_1=1$. Then the $\delta$-measure sitting at the right endpoint would be an SRB measure, and the set $C_P$ would play the role of the set of $(\phi,1)$-nontypical points $K_{\phi,1}$ (provided that the test function $\phi$ has a value $1$ on the left endpoint and $0$ on the right endpoint). But this map is discontinuous on the Cantor set! So, relying on this lemma and keeping in mind the mentioned lack of continuity, let us make the next step: construct an example of a map of the square $[0,1] \times [0,1]$ which is, in some sense, less discontinuous than the previous one, and still has a $(\phi,1)$-nontypical set of full Hausdorff dimension.

\section{Dimension 2: a sieving construction.}
Consider a square
$$
Y=\{(x,p): x,p\in [0,1]\}
$$
and describe a map $g$ on it.
Fix the ''horizontal level''
\begin{equation}\label{eq:Ip}
	I_p:=\{(x,p)\in Y \mid x\in [0,1]\}
\end{equation}
and split it into three parts:
$$I_p=I_{p,-1}\sqcup I_{p,0}\sqcup I_{p,1},$$
where

\begin{equation}\label{eq:Ipk}
\begin{aligned}
I_{p,-1}&:=[0;\frac{1-p}{2}),\\
I_{p,0}&:=[\frac{1-p}{2};\frac{1+p}{2}], \text{ and}\\
I_{p,1}&:=(\frac{1+p}{2};1]
\end{aligned}
\end{equation}
(i.e. the interval $I_{p,0}$ is the central part of 
length $p$).

To define the transformation $g$, we introduce the function $q\colon [0,1]\to [0,1]$,
\begin{equation}\label{eq:q-def}
	q(p)=p/(1+p)
\end{equation}
(so that $q(1/n)=1/(n+1)$), and mark the point $s_2=(1/2,1)$, which will be the support of SRB measure. We define $g$ separately on the different parts of every horisontal level $I_p$. For $p=1$, $g(x,p)=s_2$. For $0\le p<1$, we send the central part $I_{p,0}$ to the point $s_2$ and stretch linearly other parts $I_{p,-1}$ and $I_{p,1}$ to the whole level $I_{q(p)}$ each, see Figure~2. The shaded triangle in this figure goes to the fixed point $s_2$ under one iterate of the map $g$. This map is defined by the following formula:
\begin{equation}\label{eq:g-def}
g(x,p)= \begin{cases}
(\frac{2}{1-p}x, q(p)), &x\in I_{p,-1}\\
s_2, &x\in I_{p,0}\\
(\frac{2}{1-p}(x-\frac{1+p}{2}), q(p)), &x\in I_{p,1}.
\end{cases}
\end{equation}

\begin{center}
\scalebox{1.2} 
{
\begin{pspicture}(0,-4.5629687)(9.982813,4.5629687)
\definecolor{color108b}{rgb}{0.8,0.8,0.8}
\psframe[linewidth=0.04,dimen=outer](9.039531,3.8445315)(1.0795312,-4.115469)
\rput{-180.0}(10.12,-0.2700627){\pstriangle[linewidth=0.03,dimen=outer,fillstyle=solid,fillcolor=color108b](5.06,-4.1150312)(7.94,7.96)}
\psline[linewidth=0.04cm](1.0795312,0.4445314)(9.02,0.46296865)
\psline[linewidth=0.08cm](1.0795312,-1.4554687)(9.039531,-1.4554687)
\psline[linewidth=0.08cm](1.0995312,0.4445314)(2.7995312,0.4445314)
\psline[linewidth=0.08cm](7.3,0.44296864)(9.02,0.44296864)
\usefont{T1}{ptm}{m}{n}
\rput(0.8123438,0.4345314){$p_0$}
\usefont{T1}{ptm}{m}{n}
\rput(0.6223438,-1.4454687){$q(p_0)$}
\psdots[dotsize=0.14](5.079531,3.8445315)
\usefont{T1}{ptm}{m}{n}
\rput(5.072344,4.0745316){$s_2$}
\usefont{T1}{ptm}{m}{n}
\rput(1.9023439,0.7345314){$I_{p_0,0}$}
\usefont{T1}{ptm}{m}{n}
\rput(5.1023436,0.1945314){$I_{p_0,1}$}
\usefont{T1}{ptm}{m}{n}
\rput(8.382343,0.7345314){$I_{p_0,2}$}
\psline[linewidth=0.02cm,arrowsize=0.05291667cm 4.0,arrowlength=2.0,arrowinset=0.4]{->}(5.08,0.54296863)(5.079531,3.4645314)
\psline[linewidth=0.02cm,arrowsize=0.05291667cm 4.0,arrowlength=2.0,arrowinset=0.4]{->}(6.5,0.52296865)(5.259531,3.5045316)
\psline[linewidth=0.02cm,arrowsize=0.05291667cm 4.0,arrowlength=2.0,arrowinset=0.4]{->}(3.66,0.52296865)(4.959531,3.5045316)
\psline[linewidth=0.02cm,arrowsize=0.05291667cm 4.0,arrowlength=2.0,arrowinset=0.4]{->}(1.3195312,0.3445314)(1.48,-1.3570313)
\psline[linewidth=0.02cm,arrowsize=0.05291667cm 4.0,arrowlength=2.0,arrowinset=0.4]{->}(1.8795311,0.3645314)(4.84,-1.3970313)
\psline[linewidth=0.02cm,arrowsize=0.05291667cm 4.0,arrowlength=2.0,arrowinset=0.4]{->}(2.6395311,0.3645314)(8.12,-1.3770313)
\psline[linewidth=0.02cm,arrowsize=0.05291667cm 4.0,arrowlength=2.0,arrowinset=0.4]{->}(8.8,0.32296866)(8.58,-1.3570313)
\psline[linewidth=0.02cm,arrowsize=0.05291667cm 4.0,arrowlength=2.0,arrowinset=0.4]{->}(8.259532,0.3445314)(5.22,-1.3770313)
\psline[linewidth=0.02cm,arrowsize=0.05291667cm 4.0,arrowlength=2.0,arrowinset=0.4]{->}(7.5,0.34296864)(1.92,-1.3570313)
\psline[linewidth=0.04cm,arrowsize=0.05291667cm 4.0,arrowlength=2.0,arrowinset=0.4]{->}(1.1,3.5829687)(1.1,4.422969)
\psline[linewidth=0.04cm,arrowsize=0.05291667cm 4.0,arrowlength=2.0,arrowinset=0.4]{->}(8.58,-4.097031)(9.86,-4.097031)
\usefont{T1}{ptm}{m}{n}
\rput(0.93234366,4.3745313){$p$}
\usefont{T1}{ptm}{m}{n}
\rput(0.9323437,3.8345313){$1$}
\usefont{T1}{ptm}{m}{n}
\rput(9.012344,-4.365469){$1$}
\usefont{T1}{ptm}{m}{n}
\rput(5.012344,-4.3854685){$1/2$}
\usefont{T1}{ptm}{m}{n}
\rput(0.9723436,-4.285469){$0$}
\usefont{T1}{ptm}{m}{n}
\rput(9.6823435,-4.365469){$x$}
\end{pspicture} 
}
\\
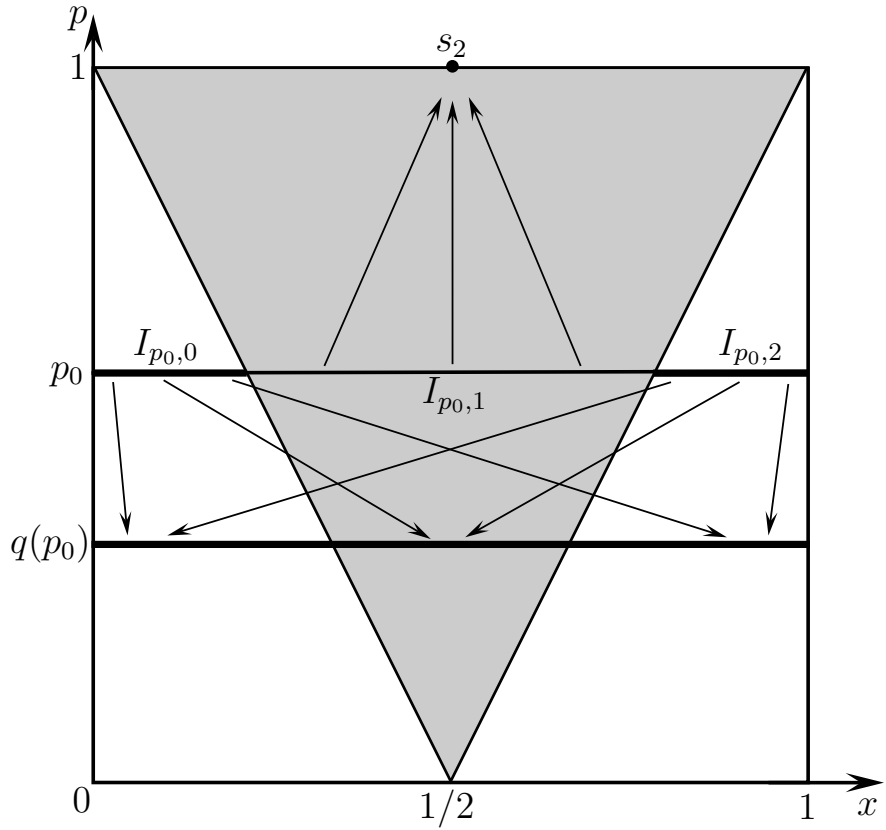
Fig. 2: the map $g$: one step of the sieving construction
\end{center}

Let
$$
K_2=\{(x,p)\mid \forall n \quad g^n(x,p)\ne s_2\}.
$$
Note that under iterates of the function $q(p)$, see~\eqref{eq:q-def} any point $p$ from $[0,1]$ tends to zero. Then our construction immediately implies that there is an alternative description of the set $K_2$:
$$
K_2=\{(x,p)\mid \dist(g^n(x,p),I_0)\to 0\}.
$$
So we have the mechanism that acts like an infinite sieve: it lets the point pass down through it to the zero level (and not to be ''thrown out'' to the point $s_2$ under the iterates of $g$) if and only if the points of the future orbit of this point never appear in the shaded triangle that is the union of the central parts of level intervals.


\begin{lemma}\label{2}
The $\delta$-measure sitting at the point $s_2$ is a global SRB measure for the map $g$. The set $K_2$ of the points that tend to the level $\{p=0\}$ under the iterates of $g$ has zero Lebesgue $2$-measure and Hausdorff dimension $2$.
\end{lemma}

\begin{proof}
First, let us prove that on every level $I_{p_0},\,0<p_0\le 1$, see~\eqref{eq:Ip}, Lebesgue almost every point is thrown out to $s_2$ under some iterate of $g$. That is, the set $D_{p_0}:=I_{p_0}\cap K_2$ has $1$-Lebesgue measure $0$. By construction, $D_{p_0}=\{p_0, x\mid x\in C_{P_0}\}$, where $P_0=P_0(p_0)=(p_0, q(p_0),\dots, q^n(p_0),\dots)$. By Lemma~\ref{1}, it is sufficient to show that $P_0\in\mathcal P$.

For $p_0=1$, $P_0=(1,1/2,1/3,1/4,\dots)\in \mathcal P$. For any $0<p_0<1$ we also have $P_0(p_0)\in\mathcal P$ due to monotonicity of the function $q$. Indeed, once $1/(n+1)<p_0\le 1/n$ for some $n\in\mathbb N$, then $1/(n+k+1)<q^k(p_0)\le 1/(n+k)$ for every $k\in\mathbb N$. Therefore, as mentioned before, the sequence $P_0$ tends to $0$, and the sum of its components is infinite, as well as for the harmonic one. Hence, $P_0$ belongs to $\mathcal P$.

The relation $\Leb_2(K_2)=0$ now follows from the Fubini theorem. Hence, $\delta$-measure at the point $s_2$ is a global SRB measure. It remains to prove that $\dim_H K_2=2$.

For every level $I_p$ denote by $\mu_p$ the measure constructed in the proof of Lemma~\ref{1} and supported on the Cantor set $D_{p_0}$ of the points that don't tend to $s_2$ under iterates of $g$. Fix any $\gamma<1$. As it was shown in the proof of Lemma~\ref{1}, all the measures $\mu_p$ for $0<p<1$ are $\gamma$-regular. Note also that the sequence of iterations $P_0(p_0)=(p_0, q(p_0),\dots, q^n(p_0),\dots)$ of the point $p_0=1$ componentwise majorises sequences of iterations for all other $p_0\in [0,1)$ due to monotonicity of the function $q$. Hence, by~\eqref{eq:lambdan}, for any $n$ the length $\lambda_n$ of intervals of the rank $n$ for $P_0=P_0(p_0)$ monotonously decreases while $p_0$ decreases, and by~\eqref{eq:regconst} the regularity constant $c=c(\gamma)$ can be chosen independently on $p$. Then the measure
$$
\mu:=\int_0^1 \mu_p dm(p)
$$
is $(\gamma+1)$-regular with the same constant $c(\gamma)$ due to Fubini theorem. 
Therefore, by Proposition~\ref{mdp}, $\dim_H S_2 \ge\gamma+1$ for every $\gamma<1$, and hence $\dim_H S_2=2$.
\end{proof}

The map $g$ is not continuous as well, but the set of discontinuity consists now only of two intervals that are the union of the vertices of the intervals $I_{p,0}$, see~\eqref{eq:Ipk} (and form two sides of the shaded triangle in Figure~2). Note also that even the image of the map $g$ is disconnected (it consists of the point $s_2$ and the lower half of the square $Y$). To get rid of these problems we include this construction into a flow. Namely, we consider such a flow that its Poincar\'e map for some cross-section resembles the map $g$ described above.



\section{Dimension 3: a flow on a stratified manifold}
In this section we present a $3$-manifold and a flow on it. The time $1$ map of this flow does not satisfy the SET. The shortcoming is that the phase space is a stratified, rather than a genuine smooth manifold. We start with their brief description. The idea is to avoid the discontinuity of the previous $2$-dimensional example by dividing the images of the parts of $I_p$ by separatrices of saddles of a smooth flow.

\subsection{Heuristic description}
We consider a closed simply connected subset $T$ of a square and multiply it directly by an interval $[0,1]$, thus obtaining a $3$-dimensional manifold with boundary, homeomorphic to a closed $3$-cube. Then we glue parts of the boundary of this manifold and obtain $2$-stratum $S_2$ (every point of this stratum has a neighborhood homeomorphic to three half balls glued together by the boundary discs). Then we define a flow such that the Poincar\'e map for the cross-section $S_2$ is very similar to the map $g$, see \eqref{eq:g-def}, while the time $1$ map of a flow is continuous and displays the same asymptotic behaviour as the Poincar\'e map. Thereinafter we take $p\in [0,1/2]$ instead of $p\in [0,1]$ to avoid incorrect constructions for $p=1$.

Now we pass to a rigorous description.

\subsection{Construction of a stratified manifold}
First we describe three curves in the square $R=\{(x,y)\mid x,y\in [-1,2]\}$. It isn't necessary to define them numerically, but just to state all their properties that we need. One can easily verify the existence of such curves.

The curve $\gamma_{-1}$ starts at the point $(0,-1)$ and ends at $(-1,0)$, and coincides with the straight intervals $\{(0,-1+t)\}$ and $\{(-1+t,0)\}$, $t\in[0,\eps]$, in some $\eps$-neighborhoods of its endpoints.
Analogously, the curve $\gamma_1$ starts at the point $(1,-1)$ ends at $(2,0)$, and coincides with the straight intervals $\{(1,-1+t)\}$ and $\{(2-t,0)\}$, $t\in[0,\eps]$, in some $\eps$-neighborhoods of its endpoints, see Figure~3a.

The curve $\gamma_0$ starts at the point $(-1,1)$ and ends at $(2,1)$. It is tangent to the line $\{(x,2)\}$ at the point $(1/2,2)$ and coincides with the straight intervals $\{(-1+t,1)\}$ and $\{(2-t,0)\}$, $t\in[0,\eps]$, in some $\eps$-neighborhoods of its endpoints.

Each of these three curves is assumed to be simple and $C^\infty$-smooth, to lie in the square $R$ and not to have any intersections with the other two curves.
For every $y_0\in (1,2)$ the curve $\gamma_0$ intersects the line $\{y=y_0\}$ at two points.

We denote by $T\subset R$ the closed subset of a square, bounded by these three curves and segments on the boundary of the square (see Figure~3a).

\begin{center}
\scalebox{1.2} 
{
 
}
\\
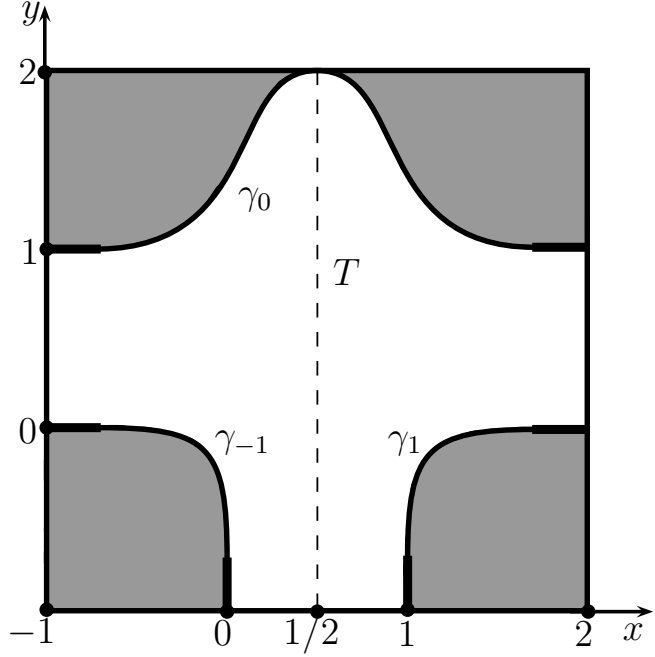
Fig. 3a: ''base'' subset $T$ of the square $R$
\end{center}

Let
\begin{equation}\label{eq:X-def}
X=T\times [0,1/2]\subset \{(x,y,p)\mid x,y\in [-1,2],p\in [0,1/2]\}.
\end{equation}

Denote left and right parts of the boundary of the level
\begin{equation}\label{eq:Tp}
T_p:=T\times \{p\}
\end{equation}
by
\begin{equation}\label{eq:Epm}
\begin{aligned}
E_{p,-1}&:=\{(-1,y,p)\mid y\in [0,1]\} \text{ and}\\
E_{p, 1}&:=\{( 2,y,p)\mid y\in [0,1]\}.
\end{aligned}
\end{equation}
We don't change the notations, thus denote the lower boundary $y=-1$ of $T_p$ by
\begin{equation}\label{eq:Ip3dim}
I_p=\{(x,-1,p)\mid x\in [0,1]\},
\end{equation}
and split it into three parts
\begin{equation}\label{eq:Ipk3dim}
\begin{aligned}
I_{p,-1}&:=\{(x,-1,p)\mid x\in [0;\frac{1-p}{2})\},\\
I_{p,0}&:=\{(x,-1,p)\mid x\in [\frac{1-p}{2};\frac{1+p}{2}]\}, \text{ and}\\
I_{p,1}&:=\{(x,-1,p)\mid x\in (\frac{1+p}{2};1]\}
\end{aligned}
\end{equation}
as in~\eqref{eq:Ip},~\eqref{eq:Ipk}
(the difference between notations is adding the coordinate $y$ with the condition $y=-1$).


We take $X$ and for every $p\in [0,1/2]$ glue linearly both intervals $E_{p,-1}$ and $E_{p,1}$ to \textit{the same} interval $I_{q(p)}$, where the function $q$ is defined in~\eqref{eq:q-def}. This equivalence is specified as follows:
\begin{equation}\label{eq:equiv3}
\begin{aligned}
(-1,y,p)\equiv(y,-1,q(p));\\
(1,y,p)\equiv(y,-1,q(p)).
\end{aligned}
\end{equation}

Thus we obtain a stratified $3$-manifold $\widetilde X$ with one $2$-dimensional stratum
\begin{equation}\label{eq:S_2}
	S_2:=\cup_{p=0}^{1/2} I_{q(p)}=\{(x,-1,p)\mid x\in[0,1],p\in[0,1/3]\},
\end{equation}
that coincides (in $\widetilde X$) with both rectangles
\begin{equation}\label{eq:Fpm}
F_{\pm}:=\cup_{p=0}^{1/2} E_{p,\pm 1},
\end{equation}
see~\eqref{eq:Epm}, according to the equivalence~\eqref{eq:equiv3}.

\subsection{Construction of a flow}
To describe the flow on $\widetilde X$, we first describe the flow on $X$.
This flow has two saddles $a_p$ and $b_p$ on any level $T_p$~\eqref{eq:Tp}. They lie on the intersection of $\gamma_{p,0}=\gamma_0 \times \{p\}$ on the level $T_p$ with the line $l_p=\{(x,y,p)\mid y=2-p/2\}$, and hence depend smoothly on $p$ and collide for $p=0$ in the point $(1/2,2,0)$.

Let
\begin{equation}\label{eq:Xpm}
\begin{aligned}
X^-:=\{(x,y,p)\in X\mid y< 2-p/2\},\\
X^+:=\{(x,y,p)\in X\mid y\ge 2-p/2\}.
\end{aligned}
\end{equation}
The set $X^-$ can be described as ''part of $X$ that lies below the plane $L$ that contains all the $a_p$'s and $b_p$'s''.

The function $p$ is the first integral of the flow restricted to $X^-$.
The flow has no singular points on $X^-$. The left and right endpoints of the interval $I_{p,0}$ (see~\eqref{eq:Ip3dim},~\eqref{eq:Ipk3dim}) lie on the incoming separatrices of $a_p$ and $b_p$ respectively (for $p=0$ these endpoints collide as well as $a_0$ and $b_0$ do). The outcoming separatrices of $a_p$ and $b_p$ for the flow restricted to $X^-_p=X_p\cap X^-$ are respective parts of $\gamma_{p,0}$. All other trajectories in $X^-_p$ go from $I_{p,\pm 1}$ to $E_{p,\pm 1}$ and from $I_{p,0}$ to interval $E_{p,0}$ on the line $l_p$ between $a_p$ and $b_p$ (see Figures~3b,~3c).

\begin{center}
\scalebox{1.2} 
{
 
}
\\
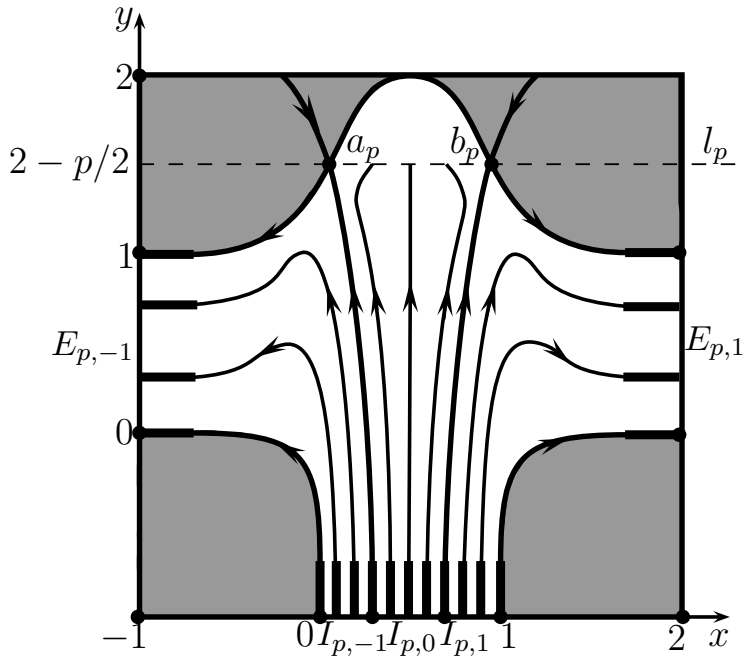
Fig. 3b: the flow on the level $T_p$ for $0<p\le 1/2$
\end{center}
\begin{center}
\scalebox{1.2} 
{
%
 
}
\\
Fig. 3c: the flow on the level $T_p$ for $p=0$
\end{center}
We also assume that there are some neighborhoods $V_\pm$ of $F_\pm$, see~\eqref{eq:Fpm}, such that the generator
of the flow is equal to $(0,\pm 1,0)$ in $V_\pm$ respectively; and a neighborhood $U$ of $S_2$ in $X$, see~\eqref{eq:S_2}, such that the the generator is equal to $(1,0,0)$ in $U$.

We have already claimed that all the trajectories in $X^-$ preserve $p$-coordinate. The latter assumptions mean that all the trajectories that escape the entrance cross-section $I_p$ are ''locally vertical'' (preserving $x$-coordinate), and all the trajectories that come to the exit cross-sections $E_{p,\pm 1}$ are ''locally horisontal'' (preserving $y$-coordinate), and, moreover, the velocity of the flow near $F_\pm$ and $S_2$ is locally constant and equal to $1$. These assumptions will allow us descend the flow easily to the glued manifold correctly ($C^\infty$-smoothly) in the $4$-dimensional case, where we deal with a genuine (non-stratified) manifold. The description of the flow restricted to $X^-$ is finished.


All trajectories in $X^+$, except for the fixed points $a_p$ and $b_p$, $p\in [0,1/2]$, and the unique sink $s_3:=(1/2,2,1/2)$, start from the plane $L$ that divides $X^+$ and $X^-$, smoothly continue the corresponding trajectories from $X^-$ that come to this plane, and tend to the sink $s_3$.

The description of the flow on $X$ is finished. We descend it on $\widetilde X$ (without any worries about its smoothness, because there is no smooth structure on the stratum $S_2$, anyway).

Let $G\colon \widetilde X\to\widetilde X$ denote its time $1$ map.

\begin{lemma}\label{3}
The $\delta$-measure sitting at the point $s_3$ is a global SRB measure for the map $G$. The set $K_3$ of the points that don't tend to $s_3$ under iterates of $G$ has zero Lebesgue $3$-measure and Hausdorff dimension $3$.
\end{lemma}

\begin{proof}
The set of the points $(x,p)$ of the ''front side'' $\{y=0\}$ that don't tend to the point $s_3$ is the intersection $K_2'$ of the set $K_2$ from the previous section with the semispace $p\le 1/2$. By Lemma~\ref{2}, the set $K_2'$ has Hausdorff dimension $2$ and zero Lebesgue $2$-measure. The set $K_3$ is a saturation of $K_2'$ by the trajectories of the flow. Therefore $K_3$ has zero Lebesgue $3$-measure by Fubini theorem and Hausdorff dimension $3$ by Proposition~\ref{mdp}. The rest of the points, obviously, tends to $s_3$.
\end{proof}

The map $G$ is continuous, but as we discussed before, after the gluing~\eqref{eq:equiv3} described above we obtain a stratified manifold. To get rid of the $2$-stratum $S_2$ we need to add one more dimension.

\section{Dimension 4: gluing up a genuine manifold}
We start with an unformal description of a flow on the $4$-manifold with a piecewise-smooth boundary, and then present the corresponding formulas.

\subsection{Heuristic description again}
Recall that $X=T\times [0,1]$ is the subset of the $3$-parallelepiped, before gluing the boundaries. Multiply it by an interval $[0,1]$, introducing new coordinate $h$.
Denote
\begin{equation}\label{eq:Mpm}
\begin{aligned}
M^-=X^-\times [0,1],\\
M^+=X^+\times [0,1].
\end{aligned}
\end{equation}
We define a flow in $M^-$ as follows: its the trajectories start on each \textit{entrance square cross-section on the level $p$} (that is, $I_p$, multiplied by the $h$-interval $[0,1]$) and preserve the $h$-coordinate until they reach the boundary. The flow is actually the same as in $X^-$ in previous section, with the condition on the generator $\dot h=0$ added.

Almost all trajectories in $M^+$ tend to a unique sink $s_4$ (as in previous section: for $M^+$ instead of $X^+$, and $s_4$ instead of $s_3$).

Then we define the gluing of the boundaries of $X\times [0,1]$ using the new coordinate $h$: the \textit{left and right hand exit square cross-sections} (that are $E_\pm$, multiplied by the $h$-interval $[0,1]$) are contracted thrice in the $h$-direction and are divorced by gluing to the \textit{separated} parts of the entrance square cross-section on the level $q(p)$.

Thus we obtain the flow that acts on a compact manifold with a piecewise smooth boundary. The time $1$ map of this flow is continuous, admits a global SRB measure sitting at the point $s_4$, and the SET doesn't hold for it. Then, by means of several simple tricks, we derive from this flow another one, which acts on the boundaryless manifold, and such that its time $1$ map is bijective.

\subsection{Construction of a manifold with a piecewise smooth boundary}
Denote the ''front'', or ''entrance'' cross-section by
$$
M_{-1}:=\{(x,-1,p,h)\mid x,h\in [0,1], p\in [0,1/2]\}.
$$
The entrance square cross-section on the level $p$
\begin{equation*}
	X_p=\{(x,-1,p,h)\mid x,h\in [0,1]\}
\end{equation*}
is divided into three parts in $x$-direction according to the splitting~\eqref{eq:Ipk}:
\begin{equation*}
\begin{aligned}
&X_{p,-1}=\{(x,0-1,p,h)\mid x\in I_{p,-1},h\in [0,1]\};\\
&X_{p,0}=\{(x,0-1,p,h)\mid x\in I_{p,0},h\in [0,1]\};\\
&X_{p,1}=\{(x,-1,p,h)\mid x\in I_{p,1},h\in [0,1]\}.
\end{aligned}
\end{equation*}
Let us also divide $X_p$ into three equal parts in $h$-direction:
\begin{equation*}
\begin{aligned}
&Z_{p,-1}=\{(x,-1,p,h)\mid x\in [0,1], h\in [0,1/3)\};\\
&Z_{p,0}=\{(x,-1,p,h)\mid x\in [0,1], h\in [1/3,2/3]\};\\
&Z_{p,1}=\{(x,0-1,p,h)\mid x\in [0,1], h\in (2/3,1]\}.
\end{aligned}
\end{equation*}

We also define the ''exit squares'' on the level $p$:
\begin{equation*}
\begin{aligned}
&A_{p,0}=\{(x,y,p,h)\mid (x,y,p)\in E_{p,0}, h\in [0,1]\},\\
&A_{p,\pm 1}=\{(x,y,p,h)\mid (x,y,p)\in E_{p,\pm 1}, h\in [0,1]\},
\end{aligned}
\end{equation*}
and glue the squares $A_{p,\pm 1}$ linearly to the rectangles $Z_{q(p),\pm 1}$ by the following equivalence:
\begin{equation}\label{eq:equiv4}
\begin{aligned}
&(-1,y,p,h)\equiv(y,-1,q(p),h/3);\\
&(2,y,p,h)\equiv(y,-1,q(p),1-(h/3)).
\end{aligned}
\end{equation}

Thus we obtain the genuine $C^\infty$-smooth manifold $M$ with a piecewise smooth. Indeed, all we need is to verify the existence of local charts on the manifold $N\subset M_{-1}$,
\begin{equation}\label{eq:N}
	N:=\{(x,-1,p,h)\mid x\in [0,1], p\in[0,1/3], h\in[0,1/3]\cup[2/3,1]\}.
\end{equation}
Note that one can continue the manifold $X\times [0,1]$ in $x$-direction to the neighborhoods of $A_{p,\pm 1}$ and in $y$-direction to the neighborhood of $Z_{q(p),\pm 1}$ (in $\mathbb R^4$), and extend the equivalence~\eqref{eq:equiv4} by the following formulas (for $\eps$ sufficiently close to $0$):
\begin{equation}\label{eq:equiv4ext}
\begin{aligned}
&(-1+\eps,y,p,h)\equiv(y,-1-\eps,q(p),h/3);\\
&(2-\eps,y,p,h)\equiv(y,-1-\eps,q(p),1-(h/3)).
\end{aligned}
\end{equation}
Then the local charts on $N$ descend from $\mathbb R^4$ to $M$ according to the latter equivalence.

\subsection{Construction of a flow}
In fact, we should add only few words to the heuristic description of the flow to make it rigorous.

As was mentioned before, the flow in $M^-$ is actually the same flow as described the $3$-dimensional case, multiplied directly by coordinate $h$ (preserving it). We assumed before, that for the flow decsribed the $3$-dimensional case, there are some neighborhoods of $F_\pm$ and $S_2$, such that the generator of the flow is equal to $(0,\pm 1,0)$ in $V_\pm$ respectively, and is equal to $(1,0,0)$ in $U$. Thus, by the equivalence~\eqref{eq:equiv4ext}, there is a neighborhood of $N$ in $M$ such that the generator of the flow is identically equal to $(1,0,0,0)$ in this neighborhood. It means that the flow is descended to $M$ $C^\infty$-smoothly.


Figure~4 displays the first return map to the ''front'' cubic cross-section $M_{-1}$
for points in $X_{p,-1}$ and $X_{p,1}$ that return to this cross-section (equivalently: for points whose $x$-coordinate is not central in terms of the map $g$, see Section~2.2). This map resembles a ''non-autonomous horseshoe map'': the $p$-levels are non-invariant and the expanding rate depends on the level.
\begin{center}
\scalebox{1.2} 
{
\begin{pspicture}(0,-4.5289063)(12.722813,4.5089064)
\definecolor{color509b}{rgb}{0.6,0.6,0.6}
\definecolor{color97b}{rgb}{0.8,0.8,0.8}
\psline[linewidth=0.08cm](1.1395311,3.4689062)(4.539531,4.4689064)
\psline[linewidth=0.08cm](8.559532,3.448906)(11.959531,4.4489064)
\psline[linewidth=0.08cm](8.559532,-3.971094)(12.039531,-3.011094)
\psline[linewidth=0.04cm,linestyle=dashed,dash=0.16cm 0.16cm,arrowsize=0.05291667cm 4.0,arrowlength=2.0,arrowinset=0.4]{->}(1.1595312,-3.971094)(5.2795315,-2.7710938)
\psline[linewidth=0.08cm](11.979531,4.4689064)(11.999531,-3.0310938)
\psline[linewidth=0.08cm](4.5195312,4.4689064)(12.019531,4.4689064)
\psline[linewidth=0.03cm,linestyle=dashed,dash=0.16cm 0.16cm](4.5195312,0.48890612)(12.019531,0.48890612)
\psline[linewidth=0.06cm](1.099531,-2.0510938)(8.599532,-2.0510938)
\psline[linewidth=0.03cm,linestyle=dashed,dash=0.16cm 0.16cm](1.1795312,-2.0310938)(4.579531,-1.031094)
\pspolygon[linewidth=0.03,fillstyle=solid,fillcolor=color509b](9.89953,0.48890612)(6.5195312,-0.53109396)(8.579531,-0.55109394)(11.93953,0.48890612)
\psline[linewidth=0.06cm](1.1395311,-0.53109396)(8.639532,-0.53109396)
\psline[linewidth=0.06cm](8.619531,-0.53109396)(11.999531,0.48890612)
\psdots[dotsize=0.16](9.879532,0.48890612)
\psdots[dotsize=0.16](11.979531,0.48890612)
\psdots[dotsize=0.16](6.539531,-0.5110939)
\psdots[dotsize=0.16](6.4195313,0.5089061)
\psdots[dotsize=0.16](1.1595312,-0.53109396)
\psdots[dotsize=0.16](8.579531,-0.55109394)
\pspolygon[linewidth=0.03,fillstyle=solid,fillcolor=color97b](10.879531,-1.371094)(3.439531,-1.371094)(4.539531,-1.031094)(11.959531,-1.031094)
\pspolygon[linewidth=0.03,fillstyle=solid,fillcolor=color509b](4.559531,0.48890612)(1.2195312,-0.5110939)(3.1195314,-0.5110939)(6.4395313,0.48890612)
\pspolygon[linewidth=0.03,fillstyle=solid,fillcolor=color97b](8.599532,-2.0310938)(1.1595312,-2.0310938)(2.279531,-1.7110939)(9.679532,-1.6910938)
\psline[linewidth=0.06cm](8.559532,-2.0310938)(11.959531,-1.031094)
\psframe[linewidth=0.08,dimen=outer](8.619531,3.488906)(1.1195312,-4.0110936)
\psdots[dotsize=0.16](3.0995314,-0.5110939)
\psdots[dotsize=0.16](4.5195312,0.5089061)
\psline[linewidth=0.06cm](4.7795315,-3.971094)(1.1795312,3.448906)
\psline[linewidth=0.06cm](4.7795315,-3.951094)(8.599532,3.488906)
\psline[linewidth=0.03cm,linestyle=dashed,dash=0.16cm 0.16cm](8.079531,-2.971094)(11.979531,4.4689064)
\psline[linewidth=0.03cm,linestyle=dashed,dash=0.16cm 0.16cm](8.079531,-2.9910936)(4.539531,4.4889064)
\psframe[linewidth=0.04,linestyle=dashed,dash=0.16cm 0.16cm,dimen=outer](12.019531,4.4889064)(4.5195312,-3.011094)
\psdots[dotsize=0.16](1.1395311,-2.0510938)
\psdots[dotsize=0.16](2.299531,-1.7110939)
\psdots[dotsize=0.16](3.439531,-1.3910939)
\psdots[dotsize=0.16](4.539531,-1.011094)
\psdots[dotsize=0.16](11.979531,-1.031094)
\psdots[dotsize=0.16](10.799531,-1.391094)
\psdots[dotsize=0.16](9.659532,-1.7110939)
\psdots[dotsize=0.16](8.579531,-2.0510938)
\psline[linewidth=0.04cm,linestyle=dotted,dotsep=0.16cm](3.0995314,-0.53109396)(3.1195314,-3.931094)
\psline[linewidth=0.04cm,linestyle=dotted,dotsep=0.16cm](6.539531,-0.55109394)(6.539531,-3.951094)
\psdots[dotsize=0.16](6.539531,-3.971094)
\psdots[dotsize=0.16](3.1195314,-3.971094)
\psdots[dotsize=0.16](4.7795315,-3.971094)
\usefont{T1}{ptm}{m}{n}
\rput(4.882344,-4.3010936){$I_{p,0}$}
\usefont{T1}{ptm}{m}{n}
\rput(7.6023436,-4.2610936){$I_{p,1}$}
\usefont{T1}{ptm}{m}{n}
\rput(2.2223437,-4.3010936){$I_{p,-1}$}
\usefont{T1}{ptm}{m}{n}
\rput(5.2923436,-3.721094){$1/2$}
\psline[linewidth=0.03cm,arrowsize=0.05291667cm 4.0,arrowlength=2.0,arrowinset=0.4]{->}(1.1595312,3.2689064)(1.1595312,4.328906)
\psline[linewidth=0.03cm,arrowsize=0.05291667cm 4.0,arrowlength=2.0,arrowinset=0.4]{->}(8.179532,-3.9910936)(9.879532,-3.971094)
\usefont{T1}{ptm}{m}{n}
\rput(9.602344,-4.1810937){$x$}
\usefont{T1}{ptm}{m}{n}
\rput(0.9723436,4.2589064){$p$}
\usefont{T1}{ptm}{m}{n}
\rput(5.0723433,-2.6210938){$h$}
\usefont{T1}{ptm}{m}{n}
\rput(2.0123436,-3.441094){$1/3$}
\psline[linewidth=0.034cm,linestyle=dotted,dotsep=0.16cm](2.299531,-1.6910938)(2.3,-3.5910938)
\psline[linewidth=0.034cm,linestyle=dotted,dotsep=0.16cm](3.44,-1.5110937)(3.439531,-3.291094)
\psdots[dotsize=0.16](2.279531,-3.6310937)
\psdots[dotsize=0.16](3.439531,-3.311094)
\usefont{T1}{ptm}{m}{n}
\rput(3.1123435,-3.181094){$2/3$}
\usefont{T1}{ptm}{m}{n}
\rput(4.3723435,-2.801094){$1$}
\usefont{T1}{ptm}{m}{n}
\rput(0.8923437,-3.9810936){$0$}
\usefont{T1}{ptm}{m}{n}
\rput(8.612343,-4.2610936){$1$}
\usefont{T1}{ptm}{m}{n}
\rput(0.9323437,3.558906){$1$}
\usefont{T1}{ptm}{m}{n}
\rput(0.85234374,-0.5010939){$p_0$}
\usefont{T1}{ptm}{m}{n}
\rput(0.6223438,-2.0610938){$q(p_0)$}
\psdots[dotsize=0.16](1.1595312,3.4689062)
\psdots[dotsize=0.16](4.559531,-2.971094)
\psdots[dotsize=0.16](8.579531,-3.971094)
\psbezier[linewidth=0.1,arrowsize=0.05291667cm 5.0,arrowlength=1.0,arrowinset=0.4]{->}(3.7395313,-0.0310939)(3.82,-0.89109373)(5.179531,-1.6710938)(5.48,-1.8910937)
\psbezier[linewidth=0.1,arrowsize=0.05291667cm 5.0,arrowlength=1.0,arrowinset=0.4]{->}(9.43953,0.0289061)(8.9,-0.87109375)(8.51953,-0.85109395)(7.74,-1.2710937)
\psbezier[linewidth=0.02,arrowsize=0.05291667cm 3.0,arrowlength=3.0,arrowinset=0.4]{->}(2.8395312,2.208906)(2.8395312,1.4089061)(3.5995314,0.4289061)(4.3595314,0.0289061)
\psbezier[linewidth=0.02,arrowsize=0.05291667cm 3.0,arrowlength=3.0,arrowinset=0.4]{->}(9.839532,2.188906)(10.259532,1.6889061)(10.519531,0.84890616)(10.219532,0.0489061)
\psbezier[linewidth=0.02,arrowsize=0.05291667cm 3.0,arrowlength=3.0,arrowinset=0.4]{->}(11.759532,-3.451094)(11.72,-2.3710938)(12.0,-1.1910938)(10.12,-1.1910938)
\psbezier[linewidth=0.02,arrowsize=0.05291667cm 3.0,arrowlength=3.0,arrowinset=0.4]{->}(10.139532,-3.711094)(9.86,-2.8710938)(9.74,-1.8310938)(8.1,-1.8510938)
\usefont{T1}{ptm}{m}{n}
\rput(10.702343,-3.8810937){$Z_{q(p_0),0}$}
\usefont{T1}{ptm}{m}{n}
\rput(11.502344,-3.601094){$Z_{q(p_0),2}$}
\usefont{T1}{ptm}{m}{n}
\rput(2.8723438,2.4789062){$X_{p_0,0}$}
\usefont{T1}{ptm}{m}{n}
\rput(9.812344,2.418906){$X_{p_0,2}$}
\psdots[dotsize=0.16](1.1595312,-3.971094)
\end{pspicture} 
}
\\
Fig. 4: first return map for the front cubic transversal $M_{-1}$
\end{center}

Analogously to the $3$-dimensional case, we put in $M^+$, see~\eqref{eq:Mpm}, the unique sink $s_4=(1/2,2,1/2,1/2)$. This sink attracts all the trajectories of the flow as soon as they reach $M^+$, except of two surfaces of the fixed points:
$$
\{(x,y,p,h)\mid (x,y)=a_p, p\in [0,1/2], h\in [0,1]\}
$$
and
$$
\{(x,y,p,h)\mid (x,y)=b_p, p\in [0,1/2], h\in [0,1]\}.
$$
 
Let $H\colon M\to M$ denote time $1$ map of the flow described above (one should distinguish it from the first return map displayed on the Figure~4). 

\begin{lemma}\label{4}
The $\delta$-measure sitting at the point $s_4$ is a global SRB measure for the map $H$. The set $K_4$ of the points that don't tend to $s_4$ under iterations of $H$ has zero Lebesgue $4$-measure and Hausdorff dimension $4$.
\end{lemma}

\begin{proof}
The set $K_4$ is the cross product of $K_3$ and an interval $[0,1]$, modulo the set of Hausdorff dimension $3$ that does not affect on it's Lebesgue $4$-measure. Hence all the assertions of the lemma follow from the corresponding statements of Lemma~\ref{3}. 
\end{proof}

\subsection{Proof of the Theorem}
\begin{proof}[Proof of the theorem]
The theorem can be easily deduced from Lemma~\ref{4}.
\textbf{Step 1.}
To construct an example of endomorphism on a $4$-manifold with a piecewise smooth boundary into itself, for which the SET fails, one should take the map $H$ that has a global SRB measure ($\delta$-measure sitting at $s_4$), and any function $\phi$ such that $\phi(s_4)=0$ and $\phi(t)=1$ $\forall t\in K_4$. Then $K_{\phi,1}=K_4$, and hence, by Lemma~4, the SET doesn't hold for $H$.

Then we make three simple improvements. First, we construct a similar example on the the manifold $\widetilde M$ with a $C^\infty$-smooth boundary.
\textbf{Step 2.}
The manifold $M$ has a piecewise smooth boundary. The non-smooth set of the boundary lies on the cross-section $M_{-1}$. It consists of the boundary (in $M_{-1}$) of the set $N$ of points that are glued by the equivalence~\eqref{eq:equiv4}, see~\eqref{eq:N}. Then one can link any $4$-regions to a small neighbourhood of this set bounded by some $C^\infty$-hypersurfaces that $C^\infty$-smoothly link to the boundary of $M$ (see Figure~5). Thus we obtain a manifold $\widetilde M$ that has a $C^\infty$-smooth boundary. The flow can be naturally extended to $\widetilde M$ by the same formula $(1,0,0,0)$ for the generator (in the natural chart in a neighborhood of $N$ in $\mathbb R^4$). Obviously, the SET still does not hold for time $1$ map $\widetilde H$ of this flow on $\widetilde M$.

\begin{center}
\scalebox{1.2} 
{
\begin{pspicture}(0,-2.65)(9.937813,2.64)
\definecolor{color898b}{rgb}{0.6,0.6,0.6}
\pscustom[linewidth=0.02,fillstyle=solid,fillcolor=color898b]
{
\newpath
\moveto(3.9078126,1.25)
\lineto(3.9178126,1.21)
\curveto(3.9228125,1.19)(3.9278126,1.145)(3.9278126,1.12)
\curveto(3.9278126,1.095)(3.9228125,1.045)(3.9178126,1.02)
\curveto(3.9128125,0.995)(3.9078126,0.945)(3.9078126,0.92)
\curveto(3.9078126,0.895)(3.9078126,0.845)(3.9078126,0.82)
\curveto(3.9078126,0.795)(3.9078126,0.745)(3.9078126,0.72)
\curveto(3.9078126,0.695)(3.9078126,0.645)(3.9078126,0.62)
\curveto(3.9078126,0.595)(3.9078126,0.545)(3.9078126,0.52)
\curveto(3.9078126,0.495)(3.9078126,0.445)(3.9078126,0.42)
\curveto(3.9078126,0.395)(3.9078126,0.345)(3.9078126,0.32)
\curveto(3.9078126,0.295)(3.9078126,0.245)(3.9078126,0.22)
\curveto(3.9078126,0.195)(3.9078126,0.145)(3.9078126,0.12)
\curveto(3.9078126,0.095)(3.9128125,0.05)(3.9178126,0.03)
\curveto(3.9228125,0.01)(3.9278126,-0.035)(3.9278126,-0.06)
\curveto(3.9278126,-0.085)(3.9278126,-0.135)(3.9278126,-0.16)
\curveto(3.9278126,-0.185)(3.9278126,-0.235)(3.9278126,-0.26)
\curveto(3.9278126,-0.285)(3.9228125,-0.33)(3.9178126,-0.35)
\curveto(3.9128125,-0.37)(3.9078126,-0.415)(3.9078126,-0.44)
\curveto(3.9078126,-0.465)(3.9078126,-0.515)(3.9078126,-0.54)
\curveto(3.9078126,-0.565)(3.9078126,-0.615)(3.9078126,-0.64)
\curveto(3.9078126,-0.665)(3.9128125,-0.715)(3.9178126,-0.74)
\curveto(3.9228125,-0.765)(3.9278126,-0.815)(3.9278126,-0.84)
\curveto(3.9278126,-0.865)(3.9328125,-0.89)(3.9378126,-0.89)
\curveto(3.9428124,-0.89)(3.9578125,-0.875)(3.9678125,-0.86)
\curveto(3.9778125,-0.845)(3.9928124,-0.805)(3.9978125,-0.78)
\curveto(4.0028124,-0.755)(4.0178127,-0.71)(4.0278125,-0.69)
\curveto(4.0378127,-0.67)(4.0528126,-0.63)(4.0578127,-0.61)
\curveto(4.0628123,-0.59)(4.0728126,-0.55)(4.0778127,-0.53)
\curveto(4.0828123,-0.51)(4.0978127,-0.475)(4.1078124,-0.46)
\curveto(4.1178126,-0.445)(4.1328125,-0.41)(4.1378126,-0.39)
\curveto(4.1428127,-0.37)(4.1578126,-0.33)(4.1678123,-0.31)
\curveto(4.1778126,-0.29)(4.1928124,-0.25)(4.1978126,-0.23)
\curveto(4.2028127,-0.21)(4.2178125,-0.175)(4.2278123,-0.16)
\curveto(4.2378125,-0.145)(4.2578125,-0.115)(4.2678127,-0.1)
\curveto(4.2778125,-0.085)(4.2978125,-0.055)(4.3078127,-0.04)
\curveto(4.3178124,-0.025)(4.3328123,0.01)(4.3378124,0.03)
\curveto(4.3428125,0.05)(4.3578124,0.085)(4.3678126,0.1)
\curveto(4.3778124,0.115)(4.4028125,0.145)(4.4178123,0.16)
\curveto(4.4328127,0.175)(4.4578123,0.205)(4.4678125,0.22)
\curveto(4.4778123,0.235)(4.4978123,0.265)(4.5078125,0.28)
\curveto(4.5178127,0.295)(4.5378127,0.325)(4.5478125,0.34)
\curveto(4.5578127,0.355)(4.5828123,0.38)(4.5978127,0.39)
\curveto(4.6128125,0.4)(4.6378126,0.425)(4.6478124,0.44)
\curveto(4.6578126,0.455)(4.6828127,0.48)(4.6978126,0.49)
\curveto(4.7128124,0.5)(4.7378125,0.525)(4.7478123,0.54)
\curveto(4.7578125,0.555)(4.7778125,0.585)(4.7878127,0.6)
\curveto(4.7978125,0.615)(4.8278127,0.64)(4.8478127,0.65)
\curveto(4.8678126,0.66)(4.9028125,0.68)(4.9178123,0.69)
\curveto(4.9328127,0.7)(4.9678125,0.72)(4.9878125,0.73)
\curveto(5.0078125,0.74)(5.0378127,0.755)(5.0478125,0.76)
\curveto(5.0578127,0.765)(5.0828123,0.78)(5.0978127,0.79)
\curveto(5.1128125,0.8)(5.1478124,0.82)(5.1678123,0.83)
\curveto(5.1878123,0.84)(5.2228127,0.865)(5.2378125,0.88)
\curveto(5.2528124,0.895)(5.2828126,0.92)(5.2978125,0.93)
\curveto(5.3128123,0.94)(5.3428125,0.96)(5.3578124,0.97)
\curveto(5.3728123,0.98)(5.4128127,0.99)(5.4378123,0.99)
\curveto(5.4628124,0.99)(5.5078125,1.0)(5.5278125,1.01)
\curveto(5.5478125,1.02)(5.5878124,1.04)(5.6078124,1.05)
\curveto(5.6278124,1.06)(5.6678123,1.075)(5.6878123,1.08)
\curveto(5.7078123,1.085)(5.7478123,1.095)(5.7678127,1.1)
\curveto(5.7878127,1.105)(5.8278127,1.115)(5.8478127,1.12)
\curveto(5.8678126,1.125)(5.9078126,1.135)(5.9278126,1.14)
\curveto(5.9478126,1.145)(5.9878125,1.155)(6.0078125,1.16)
\curveto(6.0278125,1.165)(6.0678124,1.175)(6.0878124,1.18)
\curveto(6.1078124,1.185)(6.1478124,1.195)(6.1678123,1.2)
\curveto(6.1878123,1.205)(6.2278123,1.215)(6.2478123,1.22)
\curveto(6.2678127,1.225)(6.2978125,1.235)(6.3278127,1.25)
}
\psline[linewidth=0.06cm](2.0478125,1.25)(9.907812,1.25)
\psline[linewidth=0.06cm](3.9278126,-2.49)(3.9078126,2.61)
\psbezier[linewidth=0.08](9.887813,1.23)(8.987812,1.23)(8.867812,1.25)(6.6878123,1.23)(4.5078125,1.21)(3.9078126,-0.45)(3.9278126,-1.53)(3.9478126,-2.61)(3.9278126,-1.63)(3.9278126,-2.51)
\psline[linewidth=0.03cm](4.5078125,0.25)(4.5078125,2.61)
\psline[linewidth=0.03cm](5.1078124,0.83)(5.1078124,2.61)
\psline[linewidth=0.03cm](5.7078123,1.11)(5.7078123,2.61)
\psline[linewidth=0.03cm](6.3078127,1.23)(6.3078127,2.61)
\psline[linewidth=0.03cm](6.9078126,1.25)(6.9078126,2.61)
\psline[linewidth=0.03cm](7.4878125,1.23)(7.4878125,2.61)
\psline[linewidth=0.03cm](8.087812,1.25)(8.087812,2.61)
\psline[linewidth=0.03cm](8.687813,1.27)(8.687813,2.61)
\psline[linewidth=0.03cm](9.287812,1.23)(9.287812,2.61)
\psline[linewidth=0.03cm](3.2878125,-2.45)(3.3078125,2.61)
\psline[linewidth=0.03cm](2.6878126,-2.45)(2.7078125,2.61)
\psline[linewidth=0.03cm](2.0878124,-2.45)(2.1078124,2.61)
\psline[linewidth=0.01cm,arrowsize=0.05291667cm 8.0,arrowlength=3.0,arrowinset=0.4]{->}(3.9078126,2.11)(3.9078126,2.35)
\psline[linewidth=0.01cm,arrowsize=0.05291667cm 8.0,arrowlength=3.0,arrowinset=0.4]{->}(4.5078125,0.87)(4.5078125,1.11)
\psline[linewidth=0.01cm,arrowsize=0.05291667cm 8.0,arrowlength=3.0,arrowinset=0.4]{->}(2.1078124,2.07)(2.1078124,2.31)
\psline[linewidth=0.01cm,arrowsize=0.05291667cm 8.0,arrowlength=3.0,arrowinset=0.4]{->}(2.7078125,2.07)(2.7078125,2.31)
\psline[linewidth=0.01cm,arrowsize=0.05291667cm 8.0,arrowlength=3.0,arrowinset=0.4]{->}(3.3078125,2.07)(3.3078125,2.31)
\psline[linewidth=0.01cm,arrowsize=0.05291667cm 8.0,arrowlength=3.0,arrowinset=0.4]{->}(4.5078125,2.09)(4.5078125,2.33)
\psline[linewidth=0.01cm,arrowsize=0.05291667cm 8.0,arrowlength=3.0,arrowinset=0.4]{->}(5.1078124,2.09)(5.1078124,2.33)
\psline[linewidth=0.01cm,arrowsize=0.05291667cm 8.0,arrowlength=3.0,arrowinset=0.4]{->}(5.7078123,2.09)(5.7078123,2.33)
\psline[linewidth=0.01cm,arrowsize=0.05291667cm 8.0,arrowlength=3.0,arrowinset=0.4]{->}(6.3078127,2.07)(6.3078127,2.31)
\psline[linewidth=0.01cm,arrowsize=0.05291667cm 8.0,arrowlength=3.0,arrowinset=0.4]{->}(6.9078126,2.07)(6.9078126,2.31)
\psline[linewidth=0.01cm,arrowsize=0.05291667cm 8.0,arrowlength=3.0,arrowinset=0.4]{->}(7.4878125,2.09)(7.4878125,2.33)
\psline[linewidth=0.01cm,arrowsize=0.05291667cm 8.0,arrowlength=3.0,arrowinset=0.4]{->}(8.087812,2.11)(8.087812,2.35)
\psline[linewidth=0.01cm,arrowsize=0.05291667cm 8.0,arrowlength=3.0,arrowinset=0.4]{->}(8.687813,2.09)(8.687813,2.33)
\psline[linewidth=0.01cm,arrowsize=0.05291667cm 8.0,arrowlength=3.0,arrowinset=0.4]{->}(9.287812,2.09)(9.287812,2.33)
\psline[linewidth=0.01cm,arrowsize=0.05291667cm 8.0,arrowlength=3.0,arrowinset=0.4]{->}(5.1078124,1.13)(5.1078124,1.37)
\psline[linewidth=0.01cm,arrowsize=0.05291667cm 8.0,arrowlength=3.0,arrowinset=0.4]{->}(3.9078126,0.49)(3.9078126,0.73)
\psline[linewidth=0.01cm,arrowsize=0.05291667cm 8.0,arrowlength=3.0,arrowinset=0.4]{->}(3.2878125,-0.13)(3.2878125,0.11)
\psline[linewidth=0.01cm,arrowsize=0.05291667cm 8.0,arrowlength=3.0,arrowinset=0.4]{->}(2.6878126,-0.51)(2.6878126,-0.27)
\psline[linewidth=0.01cm,arrowsize=0.05291667cm 8.0,arrowlength=3.0,arrowinset=0.4]{->}(2.0878124,-0.95)(2.0878124,-0.71)
\psbezier[linewidth=0.02,arrowsize=0.05291667cm 4.0,arrowlength=2.0,arrowinset=0.4]{->}(7.2678127,0.31)(6.6478124,0.95)(4.9078126,1.19)(4.2678127,0.67)
\psbezier[linewidth=0.02,arrowsize=0.05291667cm 4.0,arrowlength=2.0,arrowinset=0.4]{->}(5.6278124,-1.11)(4.8078127,-0.87)(4.6278124,-0.59)(4.4078126,0.05)
\psbezier[linewidth=0.02,arrowsize=0.05291667cm 4.0,arrowlength=2.0,arrowinset=0.4]{->}(0.8878125,1.75)(1.6478125,2.59)(2.1878126,1.83)(2.5678124,1.31)
\usefont{T1}{ptm}{m}{n}
\rput(0.9109375,1.66){}
\usefont{T1}{ptm}{m}{n}
\rput(0.92921877,1.32){$M_{-1}$}
\usefont{T1}{ptm}{m}{n}
\rput(6.3301563,-1.04){}
\usefont{T1}{ptm}{m}{n}
\rput(6.49,-1.32){$\partial\widetilde M$}
\usefont{T1}{ptm}{m}{n}
\rput(7.6834373,-0.22){$\widetilde M\setminus M$}
\usefont{T1}{ptm}{m}{n}
\rput(7.310625,0.2){}
\end{pspicture} 
}
\\
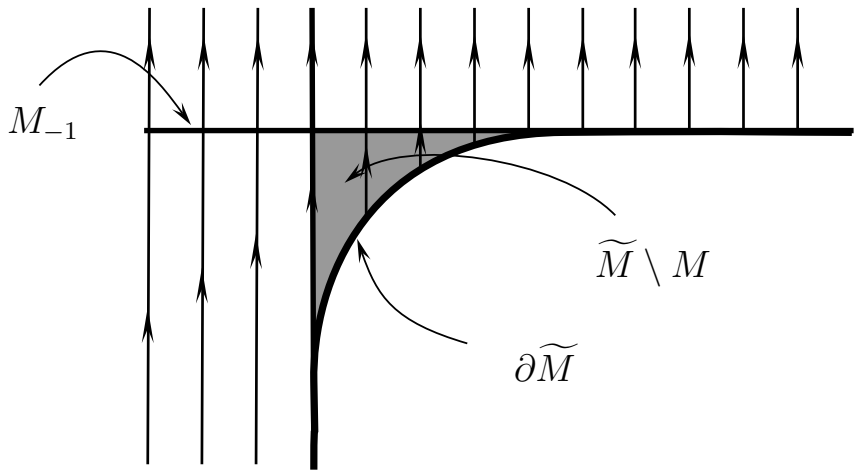
Fig. 5: slow flow on the smoothened manifold $\widetilde M$
\end{center}

\textbf{Step 3.}
The time $1$ map $\widetilde H$ of the described flow is not invertible. To make it invertible, we use the ''slow down procedure'' near the boundary of $\widetilde M$. Namely, we multiply the generator of the flow by a $C^\infty$-function 
that is equal to $0$ on the boundary $\partial \widetilde M$ of $\widetilde M$ and is positive in its interior $\interior \widetilde M$ (see Figure~5 again).

For the function $\phi$ described above, this procedure doesn't change neither the Hausdorff dimension of the $(\phi,1)$-nontypical set of the time $1$ map
of a flow, nor Lebesgue $4$-measure of the basin of attraction of $s_4$. Indeed, for every point of $\interior \widetilde M$ its positive semiorbit of the flow action does not intersect the boundary of $\widetilde M$. Therefore,
the $(\phi,1)$-nontypical set and the basin of attraction of $s_4$ for these two flows (''fast flow'' and ''slow flow'') can differ only by subsets of $\partial \widetilde M$. But $\partial \widetilde M$ has dimension $3$ and doesn't affect on the $4$-measure and on Hausdorff dimenion of the larger sets (sets of Hausdorff dimension more than $3$).

We constructed an example of a $C^\infty$-\textit{diffeomorphism} $H_{slow}$ of a compact $4$-manifold with a smooth boundary, for which the SET doesn't hold.

\textbf{Step 4.}
Note that in the previous example the boundary of $\partial \widetilde M$ totally consists of fixed points of $H_{slow}$ (one of these points, $s_4$, is a support of a global SRB measure). To obtain an example on a boundaryless manifold, we consider the double of the manifold $\widetilde M$. Namely, we take two copies of $\widetilde M$ and glue the boundaries of these copies by the natural ''identical'' map. The two points $s_4$ on the copies glue together. Our diffeomorphism $H_{slow}$ naturally extends to the diffeomorphsm $\mathcal H$ of the new manifold $\mathcal M$ (because $\partial \widetilde M$ is fixed for $H_{slow}$). Obviously, $\mathcal H$ has a global SRB measure in $\mathcal M$ (the glued point $s_4$), but the SET does not hold for it, as well as for $H_{slow}$.
\end{proof}

\subsection{Topological type of the manifold}
Our construction shows that $4$-dimensional manifold $M$ is homeomorphic to a neighborhood in $\mathbb R^4$ of union of two circles. It is also homeomorphic to the direct product of a filled pretzel and an interval. The same can be said about $\widetilde M$. Hence, the manifold $\mathcal M$ is a double of the direct product of a filled pretzel and an interval.

$\mathcal M$ can also be described as a connected sum of two $S^3\times S^1$. Indeed, the filled pretzel is a connected sum of two solid tori $D^2\times S^1$. Hence, the direct product of a filled pretzel and an interval is homeomorphic to the connected sum of two $D^3\times S^1$. But 
the doubling operation and the operation of taking a connected sum (of two equal manifolds) topologically commute. The doubling of $D^3\times S^1$ is obviously homeomorphic to $S^3\times S^1$ (as far as the doubling of every $n$-disk $D^n$ is $S^n$). Therefore, $\mathcal M$ is homeomorphic to a connected sum of two $S^3\times S^1$.

\begin{small}

\end{small}

\end{document}